\documentclass[12pt]{amsart}
\usepackage{amssymb}
\usepackage{eucal}
\usepackage{mathrsfs}
\usepackage{a4wide}

\newtheorem{theorem}{Theorem}[section]

\newtheorem{proposition}[theorem]{Proposition}
\theoremstyle{remark}

\newtheorem{question}[theorem]{Question}
\newtheorem{remark}[theorem]{Remark}

\begin{document}
\author[I. Juh\'asz]{Istv\'an Juh\'asz}
\address      {HUN-REN Alfr\'ed Rényi Institute of Mathematics%
}
\email{juhasz@renyi.hu}

\author[J. van Mill]{Jan van Mill}
\thanks
{The research in this note derives from the authors’
collaboration at the Renyi Institute in Budapest in the fall of 2025. The second-listed author is pleased to thank the Rényi Institute for
generous hospitality}
\address{University of Amsterdam}
\email{j.vanMill@uva.nl}

\author[L. Soukup]{Lajos Soukup}

\address
     {HUN-REN Alfr{\'e}d R{\'e}nyi Institute of Mathematics}
\email{soukup@renyi.hu}

\makeatletter
  \@namedef{subjclassname@2020}{%
  \textup{2020} Mathematics Subject Classification}
   \makeatother


\begin{abstract}
We show that there are locally compact spaces that can be condensed onto separable spaces, but not onto compact separable spaces.
We also show that for every cardinal $\kappa$ there is a locally compact topological group of cardinality $2^\kappa$ that can be condensed onto
a compact space but not onto a compact topological group. These answer some questions of Arhangel'skii and Buzyakova.
\end{abstract}

\title{Condensations with extra properties}

\maketitle

\section{Introduction}
\emph{All topological spaces that we discuss are infinite Tychonoff spaces.}

A \emph{condensation} $f$ of a space $X$ onto a space $Y$ is a continuous bijection $f: X\to Y$. Clearly, for a topological space $X$ admitting a continuous bijection onto a space $Y$
with a certain property is equivalent to $X$ having a coarser topology with that property. Natural examples of important coarser topologies, and hence of condensations, are for example the various weak topologies in functional analysis. Compact spaces do not have interesting condensations, but every locally compact noncompact space \emph{does}; it has a compact condensation by a classical result of Parhomenko~\cite{Parhomenko41}. There is quite an extensive literature on condensations, see for example \cite{OsipovPytkeev23, LipinOsipov22,BeluginOsipovPytkeev21} for references as well as some old and some new results.

In the recent preprint~\cite{ArhangBuz25}, Arhangel'skii and Buzyakova obtained several results on condensations for ordered spaces and their subspaces.
The main aim of this note is to answer their Question 2.7 and the compact case of Question 2.8 in the negative.

\section{Preliminaries}
For all undefined notions, we refer to Engelking~\cite{engelking:gentop} and Juh\'asz~\cite{juhasz}.

We denote the real line, the closed unit interval and the integers by $\mathbb{R}$, $\mathbb{I}$ and $\mathbb{Z}$, respectively. For a space $X$, we let $\tau X$ denote its topology. Moreover, by $X=\bigoplus_{i\in I} X_i$ we mean that $X$ is the topological sum of the spaces $\{X_i : i\in I\}$. Hence we implicitly assume that the $X_i$'s are pairwise disjoint, and
$
    \tau X = \{U\subseteq X : (\forall\, i\in I)(U\cap X_i\in\tau X_i)\}.
$
A space $X$ is \emph{homogeneous} if for all $x,y\in X$ there exists a homeomorphism $f:X\to Y$ such that $f(x)=y$. Examples of homogeneous spaces are topological groups. But there are many homogeneous spaces that do not have the structure of a topological group,  for example
the Hilbert cube since it has the fixed-point property, Keller~\cite{keller}.

As usual, we denote the cardinality of the continuum by $\mathfrak{c}$. For any cardinal number $\kappa$ we denote by $log(\kappa)$
the smallest cardinal $\mu$ such that $2^\mu \ge \kappa$.

\section{Positive results}

We present here our main results for detecting if certain topological sums
have or do not have a separable respectively a compact separable condensation.
The results in this section will be used in the next section to present our counterexamples to some questions of Arhangel'skii and Buzyakova.

\begin{theorem}\label{eerstestelling}
Assume that $X=\bigoplus \{X_\alpha : \alpha < \kappa \}$, where $$\omega \le \kappa \le 2^\mathfrak{c} \text{ and }\mu = \sup \{w(X_\alpha) : \alpha < \kappa\} \le \mathfrak{c}.$$ Then $X$ condenses onto a separable space $Y$ of weight at most $\lambda = \max\{\log(\kappa),\mu\}$.
\end{theorem}

\begin{proof}
We first consider the case $\lambda = \omega$, i.e. when $\kappa \le \mathfrak{c}$ and each $X_\alpha$ is second countable.
In this case $Y$ is to be even second countable, i.e. a subspace of the Hilbert cube $\mathbb{I}^\omega$. Now, it is obvious that
$\mathbb{I}^\omega$ can be partitioned into $\mathfrak{c}$ homeomorphic copies of $\mathbb{I}^\omega$, say
$\mathbb{I}^\omega = \bigcup \{H_\alpha : \alpha < \mathfrak{c}\}$. Then for every $\alpha < \kappa \le \mathfrak{c}$ there is an
embedding $f_\alpha$ of $X_\alpha$ into $H_\alpha$, hence the map $f = \bigcup \{f_\alpha : \alpha < \kappa\}$ is clearly
a condensation onto the second countable space $Y = \bigcup \{f_\alpha [X_\alpha] : \alpha < \kappa \}$.

Now, assume that $\lambda > \omega$.
Since $\lambda \le \mathfrak{c}$, the Tychonoff cube $\mathbb{I}^\lambda$ is separable, so
let $D = \{d_n : n < \omega\}$ be a faithfully indexed countable dense subset of $\mathbb{I}^\lambda$. Let $E \subset \lambda$ be countably infinite such that $\pi_E{\restriction}D$ is 1-1, where $\pi_E: \mathbb{I}^\mu\to \mathbb{I}^E$ is the projection.
Let $p\in  \mathbb{I}^E\setminus \pi_E(D)$. For each $n < \omega$, let $S_n = \pi_E^{-1}(\{\pi_E(d_n)\})$. Similarly, $T_p=\pi_E^{-1}(\{p\})$.
Then $\lambda > \omega$ implies that each member of $\{S_n : n < \omega\} \cup\{T_p\}$ is homeomorphic to $\mathbb{I}^\lambda$. Since $\mathbb{I}^\lambda\times\mathbb{I}^\lambda \approx\mathbb{I}^\lambda$, we can split $T_p$ into a pairwise disjoint family $\{K_\alpha : \alpha < 2^\lambda\}$ consisting of closed sets each homeomorphic to $\mathbb{I}^\lambda$. Note also that $\lambda \ge \log(\kappa)$ implies $\kappa \le 2^\lambda$.

Note that for every $\alpha < \kappa$ we have $w(X_\alpha) \le \mu \le \lambda$, hence $X_\alpha$ embeds into $\mathbb{I}^\lambda$.
For each $n < \omega$, let $f_n: X_n\to S_n$ be an embedding. Since $\mathbb{I}^\omega$ is homogeneous by Keller~\cite{keller}, so is $S_n$, hence we may assume that $f_n$ has the property that $d_n\in f_n(X_n)$.  For each $\alpha$ with $\omega \le \alpha < \kappa$ let $f_\alpha: X_\alpha\to K_\alpha$ be any embedding. Now the map
$f = \bigcup \{f_\alpha : \alpha < \kappa\}$ is clearly
a condensation onto the space $Y = \bigcup \{f_\alpha [X_\alpha] : \alpha < \kappa \} \subset \mathbb{I}^\lambda$
that is separable because $D \subset Y$.
\end{proof}

Let $X$ and $Y$ be spaces. We say that $X$ and $Y$ are \emph{somewhere homeomorphic} if there are nonempty $U\in\tau X$ and $V\in\tau Y$ such that $U$ and $V$ are homeomorphic.

\begin{theorem}\label{tweedestelling}
Let $X=\bigoplus_{n<\omega} X_n$
where each $X_n$ is $\sigma$-compact and assume that $X$ condenses onto a Baire space $Y$. Then:
\begin{enumerate}
\item there exists $n < \omega$ such that $X_n$ and $Y$ are somewhere homeomorphic;
\item if each $X_n$ is compact then there are infinitely many $n$ such that $X_n$ and $Y$ are somewhere homeomorphic;
\item if $Y$ is compact and each $X_n$ is a continuum then there are infinitely many $n$ such that $X_n$ is homeomorphic to a component of $Y$ that is clopen in $Y$.
\end{enumerate}
\end{theorem}

\begin{proof}
Let $f: X\to Y$ be a condensation.

For (1), write $X_n$ as $\bigcup_{m<\omega} X_n^m$, where each $X_n^m$ is compact. Since $Y=\bigcup_{n<\omega}\bigcup_{m<\omega} f(X_n^m)$, by the Baire Category, there are $n,m<\omega$ such that $f(X_n^m)$ contains a nonempty open subset $U$ of $Y$. Since the restriction of $f$ to $X_n^m$ is a homeomorphism, we are done.

For (2), for a fixed $n < \omega$, take $y\in f(X_{n+1})$, and let $U$ be an open neighborhood of $y$ such that $\overline{U}\cap\bigcup_{m\le n}f_m(X_m)=\emptyset$. Since $\overline{U}\subseteq\bigcup_{m>n}f_m(X_m)$, by the Baire Category Theorem, there exist $m > n$ and $V$ such that $V\subseteq \overline{U}\cap f_m(X_m)$ is nonempty and open in $\overline{U}$. Then $V\cap U$ is nonempty and open in $U$, hence open in $Y$.

For (3), first observe that by Kuratowski~\cite[Theorem 6a on page 174]{kurat:top2}, $\mathcal{C}=\{f(X_n): n < \omega\}$ is the collection of all components of $Y$. Hence $\mathcal{C}$ is an upper semicontinuous decomposition of $Y$ and if we collapse each $C\in \mathcal{C}$ to a single point, the resulting space is a countably infinite compact Hausdorff space. See Kuratowski~\cite[Theorem 4 on page 151]{kurat:top2} for details.
Being countable and compact, the decomposition space has infinitely many isolated points. But this means that infinitely many $C\in \mathcal{C}$ are clopen in $Y$.
\end{proof}

It is a natural question whether Theorem~\ref{tweedestelling} remains true for classes larger than $\sigma$-compact spaces. A natural and interesting case is that of Lindel\"of spaces. We do not know the answer to the following question.

\begin{question}
For each $n < \omega$ let $X_n$ be a nonempty Lindel\"of space of weight at most~$\mathfrak{c}$ and let $X=\bigoplus_{n<\omega} X_n$. Assume that $X$ condenses onto a separable compact space $Y$. Is there an $n< \omega$ such that $X_n$ is somewhere separable?
\end{question}

But we do know that some condition on the $X_n$'s is necessary, as we will now demonstrate.

\begin{remark}
By Dow, Gubbi and Szym\'anski~\cite[Theorem 1]{DowGubbiSzymanski}, there exists a crowded countable extremally disconnected space $X$ such that $X^*=\beta X\setminus X$ is $\omega$-bounded, that is, for every countable $A\in [X^*]^\omega$, the closure of $A$ in $\beta X$ is contained in $X^*$. Observe that $X^*$ is crowded. For if $p$ is an isolated point of $X^*$, then there is a clopen subset $C$ of $\beta X$ such that $C\cap X^*=\{p\}$. But then $\beta X$ would contain a compact countably infinite subset, hence a nontrivial convergent sequence, which contradicts $\beta X$ being extremally disconnected. By Comfort and Garci\'a-Ferreira~\cite[Theorem 6.9]{ComfortGarcia96}, $X^*$ is $\omega$-resolvable, that is, there is a partition $\{S_n : n < \omega\}$ of $X^*$ such that each $S_n$ is dense in $X^*$. Since $X^*$ is $\omega$-bounded, each $S_n$ is nowhere separable. List $X$ faithfully as $\{x_n : n < \omega\}$, and, for each $n < \omega$, put $T_n = S_n\cup\{x_n\}$. Then $\bigoplus_{n < \omega} T_n$ condenses onto the compact separable space $\beta X$, yet every $T_n$ is nowhere separable, hence in Theorem~\ref{tweedestelling} the assumption on $\sigma$-compactness cannot be dropped completely.
\end{remark}

\section{Negative results}
We first deal with Question 2.7 in Arhangel'skii and Buzyakova~\cite{ArhangBuz25}: \emph{If $X$ is locally compact space that can be condensed onto a separable space, can $X$ be condensed onto a compact separable space?}

The counterexamples to this question that we shall present will actually be topological sums
of compacta.
We actually consider topological sums of compact spaces of increasing topological complexity.

\subsection{Consistent discrete counterexamples}
To begin with, we consider discrete spaces, i.e. topological sums of points.
It does not come as a surprise that in this case we have to deal with purely set theoretic considerations.
We shall denote the discrete space of cardinality $\kappa$ by $D(\kappa)$.

Clearly, $D(\kappa)$ condenses onto a separable space iff $\omega \le \kappa \le 2^{\mathfrak{c}}$.
Moreover, $D(\kappa)$ condenses onto a separable compact space iff there is a separable compact space
of cardinality $\kappa$. In other words, our task is to determine the possible cardinalities
of separable compact spaces. This turns out to be a highly non-trivial problem that is very sensitive
to what model of set theory we are in.

Let us denote by $\bf S$ the set of possible cardinalities
of separable compact spaces. The following two conditions on $\bf S$ are then easily established in ZFC.

\smallskip

(*) $[\omega, \mathfrak{c}] \cup \{2^\kappa : \omega < \kappa \le 2^\mathfrak{c}\} \subset \bf S$;

\smallskip

(**) if $A \subset \bf S$ is countable then $\,\sup A \in \bf S$, i.e. $\bf S$ is $\omega$-closed.

\smallskip

Note that under CH, i.e. when $\mathfrak{c} = \omega_1$,
(*) reduces to $\{\omega, \omega_1, 2^{\omega_1}\} \subset \bf S$,
hence the problem is to determine $(\omega_1, 2^{\omega_1}) \cap \bf S$.

It turns out that much more can be said about $\bf S$ under CH.
This is because then $${\bf S} = \{|X| : X \text{ is compact with } w(X) \le \omega_1\}.$$
Indeed, this is immediate from Parovichenko's theorem in \cite{Parov} saying that every compact space $X$
of weight $\omega_1$ is a remainder of $D(\omega)$, hence $|X| = |X| + \omega \in \bf S$.
Also, under CH every separable space has weight $\le \omega_1$.

Since the topological sum of $\omega_1$ many spaces of weight $\le \omega_1$ is $\omega_1$,
it follows that under CH (**) can be strengthened to

\smallskip

(***) if $A \subset \bf S$ with $|A| \le \omega_1$ then $\,\sup A \in \bf S$, i.e. $\bf S$ is
both $\omega$- and $\omega_1$-closed.

\smallskip

On the other hand, compact spaces
of weight $\le \omega_1$ are just the closed subspaces of $\mathbb{I}^{\omega_1}$, and so
it is easy to see that $${\bf S} = \{|br(T)| : T \text{ is a tree with } ht(T) = |T| = \omega_1\},$$
where $br(T)$ is the set of all cofinal branches of $T$. For a closed subspace $X$ of $\mathbb{I}^{\omega_1}$
and $\alpha < \omega_1$ we have $T_\alpha = \{x \upharpoonright \alpha : x \in X\}$ as the $\alpha$th level
of the tree $T$ corresponding to $X$. $|T_\alpha| \le \omega_1$ follows from CH.

Luckily for us, $\bf S$ in this guise had been thoroughly investigated by set theorists under CH.
In fact, a complete description of $\bf S$ under CH was given by Po\' or and Shelah in \cite{PSh}.
It turns out that (*) and (***) are the only conditions that $\bf S$ has to satisfy, except that
whether $\omega_2 \in \bf S$ when $\omega_2 < 2^{\omega_1}$ requires special attention.

Indeed, as was observed by Kunen, if the "real" $\omega_2$ is a successor in the constructible universe $L$, i.e.
it is not inaccessible in $L$,  then $\omega_2 \in \bf S$, see e.g. Theorem 4.1 of \cite{PSh}.
So, in this case for any set of cardinals $S$ with $\omega_2 \in S$ and satisfying (*) and (***)
there is a forcing extension in which CH holds and $S = \bf S$.  If, on the other hand, there is an inaccessible cardinal
then in this statement the condition $\omega_2 \in \bf S$ can be dropped. This is
what the main result Theorem 3.1 in \cite{PSh} says.

Actually, for our purposes a more accessible result is Theorem 4.8(i) in \cite{Juh} due to Kunen.
This states that if there is an inaccessible cardinal then there is a model of CH in which
${\bf S} = \{\omega, \omega_1, \lambda = 2^{\omega_1}\}$, where $\lambda$ is any regular cardinal above $\omega_1$.

So, from both of these results we can deduce the consistency of many cardinals $\kappa$ with
$\mathfrak{c} = \omega_1 < \kappa < 2^{\omega_1}$ such that $D(\kappa)$ does not admit any
condensation on a separable compact space, while it does admit a
condensation on a separable compact space.

Let us now consider the case when CH fails. While in this case much less is known about the
possible behavior of $\bf S$, we can produce the consistency of having many cardinals $\kappa$ with
$\mathfrak{c} < \kappa < 2^{\mathfrak{c}}$ such that $D(\kappa)$ does not admit any
condensation on a separable compact space.

Indeed, Theorems 3.9 and 4.11 of \cite{JuhaszSoukupSzent98} imply that after adding many Cohen reals,
and thus "killing" CH, we obtain generic extensions in which $\bf S = [\omega,\mathfrak{c}] \cup \{2^\mathfrak{c}\}$,
while $2^\mathfrak{c}$ is arbitrarily higher than $\mathfrak{c}$. So again, in these forcing extensions
$D(\kappa)$  does not admit a condensation on a separable compact space whenever $\mathfrak{c} < \kappa < 2^\mathfrak{c}$.

\subsection{Nondiscrete spaces}
A discrete space is a topological sum of very uninteresting topological spaces, namely singleton sets. If in topological sums we consider more complex spaces, it is conceivable that topological arguments will become important. This is what we will see in this subsection.

Since $\log(\omega)=\omega$, by Theorems~\ref{eerstestelling} and \ref{tweedestelling}, all we need for a negative answer to \cite[Question 2.7]{ArhangBuz25} is a compact space of weight at most $\mathfrak{c}$ which is nowhere separable. Hence the \v{C}ech-Stone remainder of the discrete space $\omega$ does the job, for example. Or the countable infinite product of copies of the 1-point compactification of an uncountable discrete space of size at most $\mathfrak{c}$.
There are even first countable, compact linearly orderable spaces that are nowhere separable, even continua. Among ccc spaces, such examples are also easy to find. For example, the Stone space of the reduced measure algebra of $\mathbb{I}$, or its cone if one wants a connected example.


\subsection{Topological groups} We now turn to the compact case of Question~2.8 in Arhangel'skii and Buzyakova~\cite{ArhangBuz25}:  \emph{If $G$ is locally compact topological group that can be condensed onto a compact space, can $G$ be condensed onto a compact topological group?}

The answer to this question is in the negative. Let $G$ be a discrete space of cardinality $\aleph_\omega$. Then $G$ is a topological group and condenses onto a compact space (for example, the 1-point compactification of a discrete space of cardinality $\aleph_\omega$). But it does not condense onto a compact topological group, since the cardinality of such a group is of the form $2^\tau$ for certain $\tau$ (Comfort~\cite[Theorem 3.1]{comfort:handbook}). So perhaps what was meant in Question~2.8 is whether every topological group of the right cardinality that condenses onto a compact space, actually condenses onto a compact topological group.

\begin{proposition}\label{vierdestelling}
The topological group $\mathbb{R}$ condenses onto a compact space, but not onto a compact homogeneous space.
\end{proposition}

\begin{proof}
That $\mathbb{R}$ condenses onto a compact space follows from Parhomenko's theorem from~\cite{Parhomenko41}. In fact, it can be shown quite easily by drawing a picture that $\mathbb{R}$ condenses onto the figure eight.

Assume that $f: \mathbb{R}\to X$ is a condensation, where $X$ is any compact homogeneous space. Then $X = \bigcup_{n < \omega} f([-n,n])$, hence by the Baire Category Theorem, there exists $n < \omega$ such that $f([-n,n])$ has nonempty interior in $X$. But $f \upharpoonright [-n, n]$ is a homeomorphism, hence
this implies that there is an interval $(r,s)$ in $\mathbb{R}$ such that $f[(r,s)]$ is homeomorphic to $(r,s)$ and open in $X$.
By compactness and homogeneity of $X$, this implies that $X$ has a finite open cover consisting of sets homeomorphic to the interval $(r,s)$. This easily implies that $X$ is second countable and hence metrizable.

Thus $X$ is a metrizable 1-manifold without boundary and hence is homeomorphic to the circle (= 1-sphere) $\mathbb{S}^1$ (this is well-known, see e.g., Gale~\cite{Gale87}). Hence we may assume that $X=\mathbb{S}^1$. That this is impossible since $f$ is 1-1 is well-known and can be shown in many ways, we present an elementary one for the sake of completeness below.

Let $(u,v)$ be a nontrivial interval in $\mathbb{R}$. Then $f((u,v))$ is a nontrivial connected proper subset of $\mathbb{S}^1$, hence is either an open interval $(s,t)$, or a half-open interval $[s,t)$ (or $(s,t]$) or a closed interval $[s,t]$. We claim that the last two possibilities are impossible, so only the first one remains from which it follows that $f$ is open and hence a homeomorphism, which obviously is a contradiction. So assume that $f((u,v))$ is the interval $[s,t)$. Take $p\in (u,v)$ such that $f(p)=s$. Then $f((u,p))$ is a nontrivial connected subset of $[s,t)$ that has $p$ in its closure. Similarly for $f((p,v))$. But then $f((u,p)) \cap f((p,v))\not=\emptyset$, which contradicts $f$ being 1-1.
\end{proof}

Hence Proposition~\ref{vierdestelling} gives a counterexample of cardinality $\mathfrak{c}$. Counterexamples of cardinality $2^\tau$ for an arbitrary $\tau \ge \omega$, are also easily found. Consider such a cardinal $\tau$, and let $G$ be the topological group $\mathbb{Z}\times (\mathbb{S}^1)^\tau$. We will show that it does not condense onto a compact homogeneous space $X$. Suppose it does. Then by Theorem~\ref{tweedestelling}(3), $X$ has infinitely many clopen components homeomorphic to $(\mathbb{S}^1)^\tau$. But this contradicts homogeneity and compactness, since if there is one clopen component then all components are clopen by homogeneity and so there are only finitely many of them by compactness.


\def\cprime{$'$}
\makeatletter \renewcommand{\@biblabel}[1]{\hfill[#1]}\makeatother

\end{document}